\documentclass{scrartcl}
\usepackage[utf8]{inputenc}
\usepackage{lmodern} % font: Latin Modern
\usepackage{amsmath}
\usepackage{amssymb}
\usepackage{amsthm}
\usepackage{thmtools}
\usepackage{mathtools}
\usepackage{hyperref} % hyperlinks
\usepackage{enumitem} % advanced enumeration
\usepackage{graphicx} % for including graphics
\usepackage{dsfont} % indicator function \mathds{1}
\usepackage{tikz} % tikz pgf: advanced drawing
\usepackage[framemethod=TikZ]{mdframed}
\usepackage{ifthen}
\usepackage{supertabular} % tables with pagebreaks

\renewcommand{\subset}{\subseteq}
\newcommand{\lcb}{\left\lbrace} % '{' Left Curly Bracket
\newcommand{\rcb}{\right\rbrace} % '}' Right Curly Bracket
\newcommand{\cb}[1]{\lcb #1 \rcb} % Curly Brackets
\newcommand{\cbOf}[1]{\mathopen{}\lcb #1 \rcb\mathclose{}} % Curly Brackets
\newcommand{\lb}{\left(} %'(' Left (round) Bracket
\newcommand{\rb}{\right)} %')' Right (round) Bracket
\newcommand{\br}[1]{\lb #1 \rb} % Brackets
\newcommand{\brOf}[1]{\!\br{#1}} % Brackets
\newcommand{\abs}[1]{\left| #1 \right|} % Brackets
\newcommand*{\E}{\mathbb{E}} % expectation
\let\Pr\relax% Set equal to \relax so that LaTeX thinks it's not defined
\newcommand*{\Pr}{\mathbb{P}} % probability
\newcommand{\sizedMid}[2]{#1 \, \kern-\nulldelimiterspace\mathopen{}\left| \vphantom{#1}\,#2\right.\mathclose{}\kern-\nulldelimiterspace}

\newcommand{\Eof}[1]{\E[#1]}

\newcommand{\PrOf}[1]{\Pr\mathopen{}\lb #1 \rb\mathclose{}}
\newcommand{\Prof}[1]{\Pr(#1)}

\DeclareMathOperator{\diam}{\mathsf{diam}}

\DeclareMathOperator{\ball}{\mathrm{B}}

\providecommand\given{} % so it exists
\newcommand\SetSymbol[1][]{
	\nonscript\,#1\vert \allowbreak \nonscript\,\mathopen{}}
\DeclarePairedDelimiterX\Set[1]{\lbrace}{\rbrace}%
{ \renewcommand\given{\SetSymbol[\delimsize]} #1 }
\newcommand{\Ex}{\E\expectarg}
\DeclarePairedDelimiterX{\expectarg}[1]{[}{]}{%
	\ifnum\currentgrouptype=16 \else\begingroup\fi
	\activatebar#1
	\ifnum\currentgrouptype=16 \else\endgroup\fi
}
\newcommand{\innermid}{\nonscript\;\delimsize\vert\nonscript\;}
\newcommand{\activatebar}{%
	\begingroup\lccode`\~=`\|
	\lowercase{\endgroup\let~}\innermid 
	\mathcode`|=\string"8000
}
\newcommand*{\mc}[1]{\mathcal{#1}}

\newcommand*{\ms}[1]{\mathsf{#1}}

\newcommand*{\mf}[1]{\mathfrak{#1}}
\newcommand{\N}{\mathbb{N}}

\newcommand{\R}{\mathbb{R}}
\newcommand{\transpose}{\!^\top\!}
\newcommand{\tr}{\transpose}
\newcommand{\pr}{^\prime}

\def\integral from #1to #2of #3by #4;{\int_{#1}^{#2} \! #3 \mathrm{d}#4} %
\def\integralMeasure in #1of #2by #3of #4;{\int_{#1} \! #2{#4} #3{\mathrm{d}#4}} %
\def\mapping #1from #2to #3;{#1 \colon #2 \rightarrow #3}
\def\mappingDef #1from #2to #3maps #4to #5;{#1 \colon #2 \rightarrow #3,\ #4 \mapsto #5}
\def\seq #1by #2;{\br{#1}_{#2\in\N}}
\def\seqInText #1by #2;{(#1)_{#2\in\N}}
\newcommand{\lebesgue}{\mathcal{L}}
\newcommand{\lebesguePow}[1]{\lebesgue^{#1}}

\newcommand{\dl}{\mathrm{d}}
\def\converges for #1to #2;{\xrightarrow{#1} #2}
\def\convergesAlmostSurely for #1to #2;{\xrightarrow{#1}_{\mathsf{fs}} #2}
\def\convergesInProbability for #1to #2;{\xrightarrow{#1}_{\mathsf{p}} #2}
\def\convergesInL #1for #2to #3;{\xrightarrow{#2}_{\lebesguePow{#1}} #3}

\newcommand{\ind}{\mathds{1}}% indicator function
\newcommand{\indOf}[1]{\ind_{\!#1}}% 
\newcommand{\indOfEvent}[1]{\indOf{\cbOf{#1}}}% 
\newcommand{\normof}[1]{\Vert #1 \Vert}

\newcommand{\equationFullstop}{\, .}
\newcommand{\eqfs}{\equationFullstop}
\newcommand{\equationComma}{\, ,}
\newcommand{\eqcm}{\equationComma}

\DeclareMathOperator*{\argmin}{arg\,min}

\DeclareMathOperator*{\outerlim}{lim\, \overline{sup}}
\DeclareMathOperator*{\innerlim}{innerlim}

\def\MS{\mathcal{Q}}%the metric space
\def\epi{\mathsf{epi}}%
\def\Hausdorff{\mathsf{H}}%
\def\Haus{\mathsf{H}}%
\def\ubs{\mathsf{ubs}}%
\newcommand{\epiconv}[1]{\xrightarrow{#1\to\infty}_\epi}%
\newcommand{\ubsconv}[1]{\xrightarrow{#1\to\infty}_\ubs}%
\newcommand{\ol}[2]{\overline{#1,\!#2}}
\newcommand{\sol}[2]{\overline{#1,#2}}
\theoremstyle{plain}% Theorem-like structures provided by amsthm.sty
\newtheorem{theorem}{Theorem}[section]
\newtheorem{lemma}[theorem]{Lemma}
\newtheorem{corollary}[theorem]{Corollary}

\theoremstyle{definition}
\newtheorem{definition}[theorem]{Definition}
\newtheorem{example}[theorem]{Example}
\theoremstyle{remark}
\newtheorem{remark}[theorem]{Remark}

\newtheorem{assumptions}[theorem]{Assumptions}
\begin{document}
\title{Strong Laws of Large Numbers for Generalizations of Fréchet Mean Sets}
\subtitle{\url{https://github.com/chroetz/PaperStrong22}}
\author{Christof Schötz\\\href{mailto:math@christof-schoetz.de}{math@christof-schoetz.de}}
\date{}
\maketitle
\begin{abstract} 
A Fréchet mean of a random variable $Y$ with values in a metric space $(\mathcal Q, d)$ is an element of the metric space that minimizes $q \mapsto \mathbb E[d(Y,q)^2]$. This minimizer may be non-unique. 
%The Fréchet mean and its generalizations have an extremely wide range of application, including standard statistical methods like mean, median, or maximum likelihood estimation in euclidean spaces and nonstandard applications like finding the center of mass in Riemannian manifolds. 
We study strong laws of large numbers for sets of generalized Fréchet means.
Following generalizations are considered: 
the minimizers of $\mathbb E[d(Y, q)^\alpha]$ for $\alpha > 0$,
the minimizers of $\mathbb E[H(d(Y, q))]$ for integrals $H$ of non-decreasing functions, and
the minimizers of $\mathbb E[\mathfrak c(Y, q)]$ for a quite unrestricted class of cost functions $\mf c$.
We show convergence of empirical versions of these sets in outer limit and in one-sided Hausdorff distance. The derived results require only minimal assumptions.

\end{abstract}
\section{Fréchet Mean Sets}
For a random variable $Y$ with values in $\R^s$ and $\Eof{\normof{Y}^2}<\infty$, it holds
\begin{equation*}
	\Ex{Y} = \argmin_{q\in\R^s} \Ex{d(Y,q)^2}
	\eqcm
\end{equation*}
where $d(x,y) = \normof{x-y}$ is the Euclidean distance. We can also write 
\begin{equation*}
	\Ex{Y} = \argmin_{q\in\R^s} \Ex{d(Y,q)^2-d(Y,0)^2}
	\eqcm
\end{equation*}
as we just add a constant term. In the latter equation, we only require $Y$ to be once integrable, $\Eof{\normof{Y}}<\infty$, instead of twice as $|\normof{y-q}^2-\normof{y}^2| \leq 2\normof{y}\normof{q}+\normof{q}^2$.

The concept of Fréchet mean, proposed in \cite{frechet48}, builds upon this minimizing property of the Euclidean mean to generalize the expected value to random variables with values in a metric space. Let $(\mc Q, d)$ be a metric space. As a shorthand we may write $\ol qp$ instead of $d(q,p)$. Let $Y$ be a random variable with values in $\mc Q$. Fix an arbitrary element $o\in \mc Q$. The \emph{Fréchet mean set} of $Y$ is $M = \argmin_{q\in\mc Q} \Ex{\ol Yq^2- \ol Yo^2}$ assuming the expectations exist. 
This definition does not depend on $o$. The reason for subtracting $\ol Yo^2$ is the same as in the Euclidean case: We need to make less moment assumptions to obtain a meaningful value: The triangle inequality implies
\begin{equation*}
\abs{\ol yq^2 - \ol yo^2} \leq \ol oq \br{\ol oq + 2 \ol yo}
\end{equation*}
for all $y,q,o\in\mc Q$. Thus, if $\Ex{\ol Yo} < \infty$, then $\Ex{\ol Yq^2 - \ol Yo^2} < \infty$ for all $q\in\mc Q$.

In Euclidean spaces and other Hadamard spaces (metric spaces with nonpositive curvature), the Fréchet mean is always unique \cite[Proposition 4.3]{sturm03}. This is not true in general. On the circle, a uniform distribution on two antipodal points has two Fréchet means. For a deeper analysis of Fréchet means on the circle, see \cite{hotz15}. Similarly, Fréchet means on many positively curved spaces like (hyper-)spheres may not be unique. For the metric space $\mc Q = \R$ with $d(q,p) =\sqrt{\abs{q-p}}$, the Fréchet mean set is the set of medians, which may also be non-unique. These examples underline the importance of considering sets of minimizers in a general theory instead of assuming uniqueness.

The notion of \textit{Fréchet mean} can be generalized to cases where the cost function to be minimized is not a squared metric, e.g.\ \cite{huckemann11}. We will not explicitly write down measurablity conditions, but silently demand that all spaces have the necessary measurable structure and all functions are measurable when necessary. Let $(\mc Q, d)$ be a metric space and $\mc Y$ be a set. Let $\mf c \colon \mc Y \times \mc Q \to \R$ be a function. Let $Y$ be a random variable with values in $\mc Y$. Let $M := \argmin_{q\in\mc Q} \Ex{\mf c(Y, q)}$ assuming the expectations exist. 
In this context, $\mf c$ is called \textit{cost function}, $\mc Y$ is called \textit{data space}, $\mc Q$ is called \textit{descriptor space}, $q\mapsto\Ex{\mf c(Y, q)}$ is called \textit{objective function} (or \textit{Fréchet function}), and $M$ is called \textit{generalized Fréchet mean set} or \textit{$\mf c$-Fréchet mean set}. 

This general scenario contains the setting of general M-estimation. It includes many important statistical frameworks like maximum likelihood estimation, where $\mc Q = \Theta$ parameterizes a family of densities $(f_\vartheta)_{\vartheta\in\Theta}$ on $\mc Y = \R^p$ and $\mf c(x,\vartheta) = -\log f_{\vartheta}(x)$, or linear regression, where $\mc Q = \R^{s+1}$, $\mc Y = (\{1\}\times\R^s) \times \R$, $\mf c((x,y),\beta) = (y - \beta\tr x)^2$. It also includes nonstandard settings, e.g.\ \cite{huckemann11}, where geodesics in $\mc Q$ are fitted to points in $\mc Y$.

Fix an arbitrary element $o\in\mc Q$. We will use cost functions $\mf c(y, q) = H(\ol yq)-H(\ol yo)$, where $H(x) = \int_0^x h(t) \dl t$ for a non-decreasing function $h$, and $\mf c(y, q) = \ol yq^\alpha-\ol yo^\alpha$ with $\alpha>0$. In both cases the set of minimizers does not depend on $o$. We call the minimizers of the former cost function \textit{$H$-Fréchet means}. In the latter case, we call the minimizers \textit{power Fréchet means} or \textit{$\alpha$-Fréchet means}. We can interpret the different exponents $\alpha = 2$, $\alpha =1$, $\alpha \to 0$, $\alpha \to \infty$ as mean, median, mode, and circumcenter (or mid-range), respectively, see \cite{macqueen67}. The minimizers for $\alpha=1$ are sometimes called \textit{Fréchet median}, e.g.\ \cite{arnaudon13}. If $\mc Q$ is a Banach space, then they are called \textit{geometric} or \textit{spatial median}, e.g.\ \cite{kemperman87}.
$H$-Fréchet means serve as a generalization of $\alpha$-Fréchet means for $\alpha>1$ as well as an intermediate result for proving strong laws of large numbers for $\alpha$-Fréchet mean sets with $\alpha\in(0,1]$.

For a function $f\colon\mc Q\to\R$ and $\epsilon\geq0$, define 
\begin{equation*}
	\epsilon\text{-}\argmin_{q\in\mc Q} f(q) := \Set{q \in \mc Q \given  f(q) \leq \epsilon +\inf_{q\in\mc Q} f(q) }
	\eqfs
\end{equation*}
Let $Y_1, \dots, Y_n$ be independent random variables with the same distribution as $Y$. Choose $(\epsilon_n)_{n\in\N}\subset [0,\infty)$ with $\epsilon_n \xrightarrow{n\to\infty}0$. Let $M_n := \epsilon_n\text{-}\argmin_{q\in\mc Q} \frac1n \sum_{i=1}^n \mf c(Y_i, q)$. Our goal is to show almost sure convergence of elements in $M_n$ to elements in $M$. 
\begin{remark}
	Considering sets of elements that minimize the objective only up to $\epsilon_n$ makes the results more relevant to applications in which Fréchet mean sets are approximated numerically.		
	Furthermore, it may allow us to find more elements of $M$ in the limit of $M_n$ than for $\epsilon_n=0$, as discussed in \autoref{rem:epsilon_argmin} below and appendix \ref{sec:median}.
\end{remark}
There are different possibilities of how a convergence of sets $M_n$ to a set $M$ can be described.
\begin{definition}
	Let $(\mc Q, d)$ be a metric space.
	\begin{enumerate}[label=(\roman*)]
		\item 
		Let $(B_n)_{n\in\N}$ with $B_n \subset \mc Q$ for all $n\in\N$. Then the \emph{outer limit} of $(B_n)_{n\in\N}$ is
		\begin{equation*}
			\outerlim_{n\to\infty} B_n := \bigcap_{n\in\N} \overline{\bigcup_{k\geq n} B_k}\eqcm
		\end{equation*}
		where $\overline B$ denotes the closure of the set $B$.
		\item 
		The \emph{one-sided Hausdorff distance} between $B, B^\prime \subset \MS$ is
		\begin{equation*}
			d_\subset(B,B^\prime) := \sup_{x\in B} \inf_{x^\prime\in B^\prime} d(x,x^\prime)\eqfs
		\end{equation*}
		\item 
		The \emph{Hausdorff distance} between $B, B^\prime \subset \MS$ is
		\begin{equation*}
			d_\Hausdorff(B, B^\prime) 
			:= 
			\max(d_\subset(B,B^\prime), d_\subset(B^\prime,B))
			\eqfs
		\end{equation*}
	\end{enumerate}
\end{definition}
\begin{remark}
\mbox{ }
\begin{enumerate}[label=(\roman*)]
	\item 
		The outer limit is the set of all points of accumulation of all sequences $(x_n)_{n\in\N}$ with $x_n\in B_n$. We may write $\outerlim_{n\to\infty} B_n = \Set{q\in\mc Q \given \liminf_{n\to\infty} d(B_n, q) = 0}$, where $d(B, q) := \inf_{p\in B} d(p, q)$ for a subset $B\subset\mc Q$. The \textit{inner limit} is dual to the outer limit. It is defined as $\innerlim_{n\to\infty} B_n := \Set{q\in\mc Q \given \limsup_{n\to\infty} d(B_n, q) = 0}$. Clearly, $\innerlim_{n\to\infty} B_n \subset \outerlim_{n\to\infty} B_n$. Thus, results of the form $\outerlim_{n\to\infty} B_n \subset B$, which we show below, are stronger than $\innerlim_{n\to\infty} B_n \subset B$.
	\item 
		It holds $d_\subset(B,B^\prime)=0$ if and only if $B \subset \overline{B^\prime}$, but $d_\Hausdorff(B,B^\prime)=0$ if and only if $\overline B = \overline B^ \prime$. The function $d_\Hausdorff$ is a metric on the set of closed and bounded subsets of $\MS$.
	\item 
		Elements from a sequence of sets might have sub-sequences that have no point of accumulation and are bounded away from the outer limit of the sequence of sets. That cannot happen with the one-sided Hausdorff limit. Here, every sub-sequence is eventually arbitrarily close to the limiting set. As an example, the outer limit of the sequence of sets $\{0,n\}$, $n\in\N$ on the Euclidean real line is $\{0\}$, but $d_\subset(\{0,n\},\{0\})\xrightarrow{n\to\infty}\infty$. Aside from an element with diverging distance ($n$ in the example), another cause for the two limits to not align may be non-compactness of bounded sets: Consider the space $\ell^2$ of all sequences $(x_k)_{k\in\N}\subset\R$ with $\sum_{k=1}^\infty x_k^2 < \infty$ with distance $d((x_k)_{k\in\N}, (y_k)_{k\in\N}) = (\sum_{k=1}^\infty (x_k-y_k)^2)^\frac12$. Let $\underline{0}\in\ell^2$ be the sequence with all entries equal to 0. Let $e^n := (e^n_k)_{k\in\N}\in\ell^2$ with $e_n^n=1$ and $e^n_k=0$ for all $k\neq n$. Then $d_\subset(\{\underline 0, e_n\}, \{\underline 0\}) = 1$ for all $n\in\N$, but $\outerlim_{n\to\infty} \{\underline 0, e_n\} = \{\underline 0\}$.
\end{enumerate}
\end{remark}
We will state conditions so that $\outerlim_{n\to\infty} M_n \subset M$ almost surely or $d_\subset(M_n, M) \xrightarrow{n\to\infty}_{\ms{a.s.}} 0$, where the index $\ms{a.s.}$ indicates almost sure convergence. It is not easily possible to show $d_\Hausdorff(M_n, M) \xrightarrow{n\to\infty}_{\ms{a.s.}} 0$ if $M$ is not a singleton, as discussed in \autoref{rem:epsilon_argmin} below and appendix \ref{sec:median}. These limit theorems may be called strong laws of large numbers of the Fréchet mean set or (strong) consistency of the empirical Fréchet mean set. Notably, in \cite{evans20} the connection to convergence in the sense of topology is made:
If the set of closed subsets of $\MS$ is equipped with the \textit{Kuratowski upper topology}, a sequence of closed subsets $(B_n)_{n\in\N}$ converges to a closed subset $B$ if and only if $\outerlim_{n\to\infty} B_n \subset B$. 
If the set of nonempty compact subsets of $\MS$ is equipped with the \textit{Hausdorff upper topology}, a sequence of nonempty compact subsets $(B_n)_{n\in\N}$ converges to a nonempty compact subset $B$ if and only if $d_{\subset}(B_n, B)\xrightarrow{n\to\infty} 0$. 

\cite{ziezold77} shows a strong law in outer limit for Fréchet mean sets with a second moment condition.
\cite{sverdrup81} shows a strong law in outer limit for power Fréchet mean sets in compact spaces.
\cite{bhattacharya03} shows almost sure convergence of Fréchet mean sets in one-sided Hausdorff distance with a second moment condition.
The independent parallel work \cite{evans20} shows strong laws in outer limit and one-sided Hausdorff distance for $\alpha$-Fréchet mean sets requiring $\Ex{\ol Yo^{\alpha}} < \infty$, which is a second moment condition for the Fréchet mean.
In contrast, we show strong laws of large numbers for power Fréchet mean sets in outer limit and in one-sided Hausdorff distance with less moment assumptions: For power $\alpha>1$, we require $\Ex{\ol Yo^{\alpha-1}} < \infty$, and for $\alpha\in (0,1]$ no moment assumption is made, see \autoref{cor:cons_da} and \autoref{cor:median}. Thus, $\alpha$-Fréchet means may be of interest in robust statistics.
\cite{huckemann11} shows almost sure convergence in one-side Hausdorff distance as well as in outer limit for generalized Fréchet means. Our results for $\mf c$-Fréchet means require slightly less strict assumptions, see \autoref{thm:epi} and \autoref{thm:consistency}, which make them applicable in a larger class of settings and allows us to derive our results for $H$- and $\alpha$-Fréchet means with minimal moment assumptions.
Results in \cite{artstein95, korf01, choirat03} imply strong laws and ergodic theorems in outer limit for generalized Fréchet means. We recite parts of these results to state \autoref{thm:epi}.
Furthermore, we show strong laws of large numbers for $H$-Fréchet means sets in outer limit, \autoref{cor:nondec:epi}, and one-sided Hausdorff distance, \autoref{cor:nondec:onehaus}.
When $M$ is singleton a quantitative version (rates of convergence) of the results presented in this article is given in \cite{schoetz19}.

Before we consider the probabilistic setting, we present theory on convergence of minimizing sets for deterministic functions in section \ref{sec:det}, where we partially follow \cite{rockafellar98}. Thereafter, we derive strong laws of large numbers for $\mf c$-Fréchet mean sets in section \ref{sec:gen}, for $H$-Fréchet mean sets in section \ref{sec:nondec}, and for $\alpha$-Fréchet mean sets in section \ref{sec:power}. 
Appendix \ref{sec:median} uses the median as a simple example to illustrate some peculiarities when dealing with sets of Fréchet means.
All strong laws in the main part of this article build upon \cite[Theorem 1.1]{korf01} -- a deep convergence result for functions $\MS \to \R$. In appendix \ref{sec:alt}, we show a different route to a strong law in one-side Hausdorff distance. This is illustrative, but requires slightly stricter assumptions. 
In appendix \ref{sec:aux} some auxiliary results are stated and proven.
\section{Convergence of Minimizer Sets of Deterministic Functions}\label{sec:det}
Let $(\mc Q, d)$ be a metric space. The diameter of a set $B\subset \mc Q$ is defined as $\diam(B) = \sup_{q,p\in B} d(q,p)$. The following notion of convergence of functions will be useful to infer a convergence results of their minimizers.
\begin{definition}
	Let $f,f_n \colon \MS\to\R$, $n\in\N$.
	The sequence $(f_n)_{n\in\N}$ \emph{epi-converges} to $f$ \emph{at} $x\in\MS$ if and only if
	\begin{align*}
	\forall (x_n)_{n\in\N} \subset \MS, x_n \to x \colon& \liminf_{n\to\infty} f_n(x_n) \geq f(x)\qquad\text{and}\\
	\exists (y_n)_{n\in\N} \subset \MS, y_n \to x \colon& \limsup_{n\to\infty} f_n(y_n) \leq f(x)
	\eqfs
	\end{align*}
	The sequence $(f_n)_{n\in\N}$ \emph{epi-converges} to $f$ if and only if it epi-converges at all $x\in\MS$. We then write $f_n\epiconv{n}f$.
\end{definition}
We introduce some short notation. Let $f \colon \MS \to \R$ and $\epsilon\geq 0$. Denote $\inf f = \inf_{x\in\MS}f(x)$, $\argmin f = \Set{x \in\MS \given f(x) = \inf f}$, $\epsilon\text{-}\argmin f = \Set{x\in\MS \given  f(x) \leq \epsilon +\inf f }$.
Let $\delta > 0$ and $x_0\in\mc Q$.
Denote $\ball_\delta(x_0) = \Set{x\in\MS \given  d(x, x_0) < \delta}$.
%  and $\ball_\delta(A) = \bigcup_{x\in A} \ball_\delta(x)$. 
Furthermore, $f$ is called \emph{lower semi-continuous} if and only if  $\liminf_{x\to x_0} f(x) \geq f(x_0)$ for all $x_0\in\mc Q$. 

To state convergence results for minimizing sets of deterministic functions, we need one final definition.
\begin{definition}
	A sequence $(B_n)_{n\in\N}$ of sets $B_n \subset \MS$ is called \emph{eventually precompact} if and only if there is $n\in\N$ such that the set 
$\bigcup_{k=n}^\infty B_k$ is precompact, i.e. its closure is compact.
\end{definition}
The first theorem of this section relates epi-convergence of functions to convergence of their sets of minimizers in outer limit.
\begin{theorem}\label{thm:convOfMini}
	Let $f,f_n \colon \MS\to\R$. Let $(\epsilon_n)_{n\in\N}\subset [0,\infty)$ with $\epsilon_n \xrightarrow{n\to\infty}0$. 
		Assume $f_n\epiconv{n}f$.
		Then
		\begin{equation*}
			\outerlim_{n\to\infty}\, \epsilon_n\text{-}\argmin f_n \subset \argmin f
		\end{equation*}
		and
		\begin{equation*}
			\limsup_{n\to\infty} \inf f_n \leq \inf f
			\eqfs
		\end{equation*}
\end{theorem}
Large parts of this theorem can be found e.g., in \cite[chapter 7]{rockafellar98}. To make this article more self-contained, we give a proof here.
\begin{proof}
		Let $x \in \outerlim_{n\to\infty}\, \epsilon_n\text{-}\argmin f_n$. Then there is a sequence $x_n \in \epsilon_{n}\text{-}\argmin f_{n}$ with a subsequence converging to $x$, i.e., $x_{n_i} \xrightarrow{i\to\infty} x$, where $n_i\xrightarrow{i\to\infty}\infty$.
		Let $y\in\MS$ be arbitrary. As $f_n\epiconv{n}f$, there is a sequence $(y_n)_{n\in\N}\subset\MS$ with $y_n\xrightarrow{n\to\infty}y$ and $\limsup_{n\to\infty} f_n(y_n) \leq f(y)$. 
		It holds $f_{n_i}(x_{n_i}) \leq \epsilon_{n_i} + \inf f_{n_i} \leq \epsilon_{n_i} + f_{n_i}(y_{n_i})$. Thus, by the definition of epi-convergence and $\epsilon_n \xrightarrow{n\to\infty} 0$, we obtain
		\begin{equation*}
		f(x) 
		\leq 
		\liminf_{i\to\infty} f_{n_i}(x_{n_i}) 
		\leq 
		\liminf_{i\to\infty} \br{\epsilon_{n_i} + f_{n_i}(y_{n_i})}
		\leq 
		\limsup_{i\to\infty} f_{n_i}(y_{n_i}) 
		\leq 
		f(y)\eqfs
		\end{equation*}		
		Thus, $x \in \argmin f$.
		Next, we turn to the inequality of the infima. For $\epsilon > 0$ choose an arbitrary $x\in\epsilon\text{-}\argmin f$. There is a sequence $(y_n)_{n\in\N}\subset\MS$ with $y_n\xrightarrow{n\to\infty} x$ and $f_n(y_n)\xrightarrow{n\to\infty} f(x)$. Thus,
		\begin{equation*}
			\limsup_{n\to\infty} \inf f_n \leq \limsup_{n\to\infty} f_n(y_n) \leq \inf f + \epsilon
			\eqfs
		\end{equation*}
\end{proof}
It is illustrative to compare this result with \autoref{thm:alt:convOfMini}, which shows that a stronger notion of convergence for functions -- convergences uniformly on bounded sets -- yields convergence of sets of minimizers in one-sided Hausdorff distance, which is a stronger notion of convergence of sets as the next theorem shows.
\begin{theorem}\label{thm:outerVsHaus}
	Let $(B_n)_{n\in\N}$ with $B_n \subset \mc Q$ for all $n\in\N$. Let $B \subset \MS$.
	\begin{enumerate}[label=(\roman*)]
		\item If $d_{\subset}(B_n, B) \xrightarrow{n\to\infty} 0$ then $\outerlim_{n\to\infty}\, B_n \subset \overline{B}$.
		\item Assume $(B_n)_{n\in\N}$ is eventually precompact. If $\outerlim_{n\to\infty}\, B_n \subset \overline{B}$ then $d_{\subset}(B_n, B) \xrightarrow{n\to\infty} 0$.
	\end{enumerate}	
\end{theorem}
\begin{proof}
	\mbox{ }
	\begin{enumerate}[label=(\roman*)]
\item 
	Assume $d_{\subset}(B_n, B) \xrightarrow{n\to\infty} 0$. Let $x_\infty \in \outerlim_{n\to\infty}\, B_n$, i.e., there is a sequence $(x_{n_k})_{k\in\N}\subset \MS$ with $n_1 < n_2 < \dots$ and $x_{n_k}\in B_{n_k}$ such that $x_{n_k} \xrightarrow{k\to\infty} x_\infty$. Thus, 
	\begin{equation*}
		\inf_{x\in B} d(x_\infty, x) \leq d(x_\infty, x_{n_k}) + \inf_{x\in B} d(x_{n_k}, x)  \xrightarrow{k\to\infty} 0\eqfs
	\end{equation*}
	This shows $\inf_{x\in B} d(x_\infty, x) = 0$. Hence,
	\begin{equation*}
		\outerlim_{n\to\infty}\,B_n \subset \Set{x_\infty \in \MS \given \inf_{x\in B} d(x_\infty, x) = 0} = \overline{B}\eqfs
	\end{equation*}
\item	
	Assume  $\outerlim_{n\to\infty}\, B_n \subset \overline{B}$. Further assume the existence of $\epsilon>0$ and a sequence $(x_{n_k})_{k\in\N}\subset \MS$ with $n_1 < n_2 < \dots$ and $x_{n_k}\in B_{n_k}$ such that $\inf_{x\in B} d(x_{n_k}, x) \geq \epsilon$. As $(B_{n_k})_{k\in\N}$ is eventually precompact, the sequence $(x_{n_k})_{k\in\N}$ has an accumulation point $x_\infty$ in $\overline{\bigcup_{k \geq k_0} B_{n_k}}$ for some $k_0\in\N$ with $\inf_{x\in B} d(x_{\infty}, x) \geq \epsilon$. In particular, $x_{\infty} \not\in \overline{B}$, which contradicts the first assumption in the proof. Thus, a sequence $(x_{n_k})_{k\in\N}$ with these properties cannot exist, which implies $d_{\subset}(B_n, B) \xrightarrow{n\to\infty} 0$.
	\end{enumerate}
\end{proof}
Note that the argument for the second part is essentially the same as in \cite[proof of Theorem A.4]{huckemann11}.
\begin{remark}\label{rem:epsilon_argmin}
Together \autoref{thm:convOfMini} and \autoref{thm:outerVsHaus} may yield convergence of minimizers in one-sided Hausdorff distance. But even if $d_{\subset}(\epsilon_n\text{-}\argmin f_n, \argmin f) \xrightarrow{n\to\infty} 0$, $d_\Haus(\epsilon_n\text{-}\argmin f_n, \argmin f)$ does not necessarily vanish unless $\argmin f$ is a singleton.
Similarly, for an arbitrary sequence $\epsilon_n \xrightarrow{n\to\infty}0$, the outer limit of $\epsilon_n\text{-}\argmin f_n$ may be a strict subset of $\argmin f$.
But according to \cite[Theorem 7.31 (c)]{rockafellar98}, there exists a sequence $(\epsilon_n)_{n\in\N}$ with $\epsilon_n\xrightarrow{n\to\infty}0$ slow enough such that $\outerlim_{n\to\infty}\epsilon_n\text{-}\argmin f_n = \argmin f$. An explicit example of this phenomenon is presented in appendix \ref{sec:median}.
\end{remark}
\section{Strong Laws for $\mf c$-Fréchet Mean Sets}\label{sec:gen}
Let $(\mc Q, d)$ be a metric space, the descriptor space. Let $\mc Y$ be a set, the data space. Let $\mf c\colon \mc Y \times\mc Q \to \R$ be a function, the cost function. Let $(\Omega, \Sigma, \Pr)$ be a probability space that is silently underlying all random variables in this section. Let $Y$ be a random variable with values in $\mc Y$. Denote the $\mf c$-Fréchet mean set of $Y$ as $M = \argmin_{q\in\mc Q} \Ex{\mf c(Y, q)}$.
Let $Y_1, \dots, Y_n$ be independent random variables with the same distribution as $Y$.
Choose $(\epsilon_n)_{n\in\N}\subset [0,\infty)$ with $\epsilon_n \xrightarrow{n\to\infty}0$. Set $M_n = \epsilon_n\text{-}\argmin_{q\in\mc Q} \frac1n \sum_{i=1}^n \mf c(Y_i, q)$.
\begin{assumptions}\mbox{ }
\begin{itemize}
\item 
	\textsc{Polish}: 
	$(\mc Q, d)$ is separable and complete.
\item 
	\textsc{LowerSemiContinuity}: 
	$q \mapsto \mf c(y, q)$ is lower semi-continuous.
\item 
	\textsc{Integrable}: 
	$\Ex{\abs{\mf c(Y, q)}} < \infty$ for all $q\in\mc Q$.
\item 
	\textsc{IntegrableInf}: 
	$\Ex{\inf_{q \in \mc Q}\mf c(Y, q)} > - \infty$.
\end{itemize}
\end{assumptions}
\begin{theorem}\label{thm:epi}
	Assume \textsc{Polish}, \textsc{LowerSemiContinuity}, \textsc{Integrable}, and \textsc{IntegrableInf}.
	Then, almost surely,
	\begin{equation*}
		\outerlim_{n\to\infty}\, M_n \subset M
		\eqfs
	\end{equation*}
\end{theorem}
\begin{proof}
Define $F(q) = \Ex{\mf c(Y, q)}$, $F_n(q) = \frac1n\sum_{i=1}^n \mf c(Y_i, q)$. By \textsc{Integrable}, $F(q)<\infty$.
\cite[Theorem 1.1]{korf01} states that $F_n \xrightarrow{n\to\infty}_{\ms{epi}} F$ almost surely if \textsc{Polish}, \textsc{LowerSemiContinuity}, and \textsc{IntegrableInf} are true. \autoref{thm:convOfMini} then implies $\outerlim_{n\to\infty}\, M_n \subset M$ almost surely.
\end{proof}
\begin{assumptions}\mbox{ }
\begin{itemize}
	\item \textsc{HeineBorel}: Every closed bounded set in $\mc Q$ is compact.
	\item \textsc{SampleHeineBorel}: Almost surely following is true: There is $N_0\in\N$ such that every closed and bounded subset of $\bigcup_{n\geq N_0} M_n$ is compact.
	\item \textsc{UpperBound}: $\Ex{\sup_{q\in B} \abs{\mf c(Y,q)}} < \infty$ for all bounded sets $B\subset\mc Q$.
	\item \textsc{LowerBound}: There are $o\in\mc Q$, $\psi^+, \psi^- \colon [0,\infty) \to [0,\infty)$, $\mf a^+, \mf a^- \in(0,\infty)$, and $\sigma(Y_1, \dots, Y_n)$-measurable random variables $\mf a^+_n, \mf a^-_n \in[0,\infty)$ such that 
	\begin{align*}
		\mf a^+ \psi^+(\ol qo) - \mf a^- \psi^-(\ol qo) &\leq \Ex{\mf c(Y, q)}
		\eqcm\\
		\mf a^+_n \psi^+(\ol qo) - \mf a^-_n \psi^-(\ol qo) &\leq \frac1n \sum_{i=1}^n \mf c(Y_i, q)
	\end{align*}
	for all $q\in\mc Q$. Furthermore, $\mf a^+_n\xrightarrow{n\to\infty}_{\ms{a.s.}}\mf a^+$ and $\mf a^-_n\xrightarrow{n\to\infty}_{\ms{a.s.}}\mf a^-$. Lastly, $\psi^+(\delta)/\max(1, \psi^-(\delta))\xrightarrow{\delta\to\infty}\infty$.\footnote{A previous version omitted the $\max(1,\cdot)$, which was an error. I became aware of this mistake through work by Jaesung Park and Sungkyu Jung, whom I thank for their careful analysis.}
	\end{itemize}
\end{assumptions}
\begin{remark}\label{rem:on_ass1}\mbox{ }
\begin{itemize}
\item 
	Following implications hold:
	\begin{align*}
	\textsc{HeineBorel} &\Rightarrow \textsc{Polish}\eqcm\\
	\textsc{HeineBorel} &\Rightarrow \textsc{SampleHeineBorel}\eqcm\\
	\textsc{UpperBound} &\Rightarrow \textsc{Integrable}\eqfs
	\end{align*}
\item On \textsc{HeineBorel}:
	A space enjoying this property is also called \textit{boundedly compact} or \textit{proper} metric space.
	The Euclidean spaces $\R^s$, finite dimensional Riemannian manifolds, as well as $\mc C^\infty(U)$ for open subsets $U \subset \R^s$  fulfill \textsc{Heine--Borel} \cite[section 8.4.7]{edwards95}.
	See \cite{williamson87} for a construction of further spaces where \textsc{Heine--Borel} is true. 
\item On \textsc{SampleHeineBorel} and infinite dimension:
	If $M_n=\{m_n\}$ and $M=\{m\}$ are singleton sets and $m_n \xrightarrow{n\to\infty} m$ almost surely, then \textsc{SampleHeineBorel} holds. It is less strict than \textsc{HeineBorel}: In separable Hilbert spaces of infinite dimension \textsc{HeineBorel} does not hold. But with the metric $d$ induced by the inner product and $\mf c=d^2$, $\mf c$-Fréchet means are unique and equal to the usual notion of mean. Furthermore, strong laws of large numbers in Hilbert spaces are well-known, see e.g.\ \cite{kawabe86}. Thus, \textsc{SampleHeineBorel} is true. Let it be noted that proving \textsc{SampleHeineBorel} in a space where \textsc{HeineBorel} is false may be of similar difficulty as showing convergence of Fréchet means directly. In the case of infinite dimensional Banach spaces, results on strong laws of large numbers for a different notion of mean -- the Bochner integral -- are well established, see e.g.\ \cite{hoffmann76}.
\item On \textsc{LowerBound}:
	We illustrate this condition in the linear regression setting with $\mc Q = \R^{s+1}$, $\mc Y = (\{1\}\times\R^s) \times \R$, $\mf c((x,y),\beta) = (y - \beta\tr x)^2 - y^2 = - 2 \beta\tr x y + \beta\tr xx\tr \beta$. Let $(X, Y)$ be random variables with values in $\mc Y$. Let $(X_1, Y_1),\dots, (X_n, Y_n)$ be independent with the same distribution as $(X,Y)$. We can set $o = 0\in\R^{s+1}$, $\mf a^+ = \lambda_{\ms{min}}(\Ex{XX\tr})$, where $\lambda_{\ms{min}}$ denotes the smallest eigenvalue, $\mf a^- = 2 \normof{\Ex{XY}}$, $\mf a^+_n = \lambda_{\ms{min}}(\frac1n \sum_{i=1}^n X_iX_i\tr)$, $\mf a^-_n = 2 \normof{\frac1n\sum_{i=1}^n X_iY_i}$, $\psi^+(\delta) = \delta^2$ and $\psi^-(\delta) = \delta$. If $\lambda_{\ms{min}}(\Ex{XX\tr}) > 0$, the largest eigenvalue $\lambda_{\ms{max}}(\Ex{XX\tr}) < \infty$, and $\Ex{\normof{XY}} < \infty$, all conditions are fulfilled.
	
	For a further application of \textsc{LowerBound}, see the proof of \autoref{cor:nondec:onehaus} in the next section.
\end{itemize}
\end{remark}
\begin{theorem}\label{thm:consistency}
	Assume \textsc{Polish}, \textsc{LowerSemiContinuity}, \textsc{IntegrableInf},  \textsc{SampleHeineBorel}, \textsc{UpperBound}, and \textsc{LowerBound}. Then
	\begin{equation*}
		d_\subset(M_n, M) \xrightarrow{n\to\infty}_{\ms{a.s.}} 0
		\eqfs
	\end{equation*}
\end{theorem}
\begin{proof}
The proof consists of following steps:
\begin{enumerate}
\item Apply \autoref{thm:epi}.
\item Reduction to a bounded set.
\item Show that $M_n$ is eventually precompact almost surely.
\item Apply \autoref{thm:outerVsHaus}.
\end{enumerate}
\noindent
\underline{\smash{Step 1.}}
\textsc{Polish}, \textsc{LowerSemiContinuity}, and \textsc{IntegrableInf} are assumptions. \textsc{UpperBound} implies \textsc{Integrable}. Thus, \autoref{thm:epi} yields $\outerlim_{n\to\infty}\, M_n \subset M$ almost surely.

\noindent
\underline{\smash{Step 2.}}
Define $F(q) = \Ex{\mf c(Y, q)}$, $F_n(q) = \frac1n\sum_{i=1}^n \mf c(Y_i, q)$. 
We want to show that there is a bounded set $B_1 \subset \mc Q$ such that $F(q) \geq F(m) + 1$ and $F_n(q) \geq F_n(m) + 1$ for all $q\in \mc Q \setminus B_1$ and $m \in M$. If $\mc Q$ is bounded, we can take $B_1 = \mc Q$. Assume $\mc Q$ is not bounded. 

Let $m \in M$. By \textsc{UpperBound}, $F(m) < \infty$. Let $o\in\mc Q$ from \textsc{LowerBound}.
Due to \textsc{LowerBound}, $F(q) \geq \mf a^+ \psi^+(\delta) - \mf a^- \psi^-(\delta) \geq  F(m) + 2$ for all $q\in\mc Q\setminus \ball_\delta(o)$ and $\delta$ large enough. This holds for all $m\in M$ as $F(m)$ does not change with $m$. We set $B_1 = \ball_\delta(o)$. For $F_n$, it holds $F_n(m) \xrightarrow{n\to\infty}_{\ms{a.s.}} F(m)$ and $\inf_{q\in\mc Q \setminus B_1} F_n(q) \geq \mf a^+_n \psi^+(\delta) - \mf a^-_n \psi^-(\delta)$ with $\mf a^+_n \xrightarrow{n\to\infty}_{\ms{a.s.}} \mf a^+$ and $\mf a^-_n \xrightarrow{n\to\infty}_{\ms{a.s.}} \mf a^-$.
Thus, there is a random variable $N_1$ such that almost surely $F_n(q) \geq F_n(m) + 1$ for all $n \geq N_1$, $q\in \mc Q \setminus B_1$, and $m\in M$.

\noindent
\underline{\smash{Step 3.}}
Take $N_0$ from \textsc{SampleHeineBorel}. Choose $N_2\geq \max(N_0, N_1)$ such that $\epsilon_n < 1$ for all $n \geq N_2$. Then $M_n \subset B_1$ for all $n\geq N_2$. Thus, $\bigcup_{n\geq N_2}M_n$ is bounded and -- due to \textsc{SampleHeineBorel}, \textsc{Polish}, and \autoref{lmm:precompact} -- precompact almost surely. 

\noindent
\underline{\smash{Step 4.}}
Finally, step 1 and 3 together with \autoref{thm:outerVsHaus} yield $d_\subset(M_n, M) \xrightarrow{n\to\infty}_{\ms{a.s.}} 0$.
\end{proof}
\section{Strong Laws for $H$-Fréchet Mean Sets}\label{sec:nondec}
Let $(\mc Q, d)$ be a metric space. Let $(\Omega, \Sigma, \Pr)$ be a probability space that is silently underlying all random variables in this section. Let $Y$ be a random variable with values in $\mc Q$.
Let $h \colon [0,\infty) \to [0,\infty)$ be a non-decreasing function.
Define $H \colon [0,\infty) \to [0,\infty), x\mapsto \int_0^x h(t) \dl t$.
Fix an arbitrary element $o\in\mc Q$. Denote the $H$-Fréchet mean set of $Y$ as $M = \argmin_{q\in\mc Q} \Ex{H(\ol Yq) - H(\ol Yo)}$.
Let $Y_1, \dots, Y_n$ be independent random variables with the same distribution as $Y$.
Choose $(\epsilon_n)_{n\in\N}\subset [0,\infty)$ with $\epsilon_n \xrightarrow{n\to\infty}0$. Set $M_n = \epsilon_n\text{-}\argmin_{q\in\mc Q} \frac1n \sum_{i=1}^n (H(\ol {Y_i}q) - H(\ol {Y_i}o))$.
\begin{assumptions}\mbox{ }
	\begin{itemize}
	\item \textsc{InfiniteIncrease}:
		$h(x) \xrightarrow{x\to\infty} \infty$.
	\item \textsc{Additivity}:
		There is $b \in [1,\infty)$ such that $h(2x) \leq b h(x)$ for all $x \geq 0$.
	\item \textsc{$h$-Moment}: 
		$\Ex{h(\ol Yo)} < \infty$.
	\end{itemize}
\end{assumptions}
\begin{remark}\mbox{ }
\begin{itemize}
\item On \textsc{Additivity}:
	This implies $h(x+y)\leq b(h(x) + h(y))$ for all $x,y\geq 0$, see \autoref{lmm:nondec} (appendix).
	If $h$ is concave, \textsc{Additivity} holds with $b=2$ and we even have $h(x+y)\leq h(x) + h(y)$.
	This condition is not very restrictive, but it excludes functions that grow exponentially.
\end{itemize}
\end{remark}
\begin{corollary}\label{cor:nondec:epi}
	Assume \textsc{Polish}, \textsc{Additivity}, and \textsc{$h$-Moment}.
	Then, almost surely, 
	\begin{equation*}
		\outerlim_{n\to\infty}\, M_n \subset M
		\eqfs
	\end{equation*}
\end{corollary}
\begin{proof}
	We check the conditions of \autoref{thm:epi}. \textsc{Polish} is an assumption. 
	\textsc{LowerSemiContinuity} is fulfilled as $(q,p)\mapsto d(q,p)$ and $x \mapsto H(x)$ are continuous.
	For \textsc{Integrable}, we note that $H$ is non-decreasing and apply \autoref{lmm:nondec} (i),
	\begin{align*}
		\abs{H(\ol yq) - H(\ol yo)} 
		&\leq 
		\abs{\ol yq - \ol yo} h\brOf{\max(\ol yq, \ol yo)}
		\\&\leq 
		\ol qo\, h(\ol qo + \ol yo)
		\\&\leq 
		b\, \ol qo \br{h(\ol qo) + h(\ol yo)}
		\eqcm
	\end{align*}
	where the last inequality follows from \autoref{lmm:nondec} (ii) using \textsc{Additivity}.
	Thus, \textsc{$h$-Moment} implies \textsc{Integrable}.
	To show \textsc{IntegrableInf}, we note that $H$ is non-decreasing and apply \autoref{lmm:nondec} (iii),
	\begin{align*}
		H(\ol yq) - H(\ol yo)
		&\geq 
		H(\abs{\ol yo - \ol qo}) - H(\ol yo)
		\\&\geq 
		b^{-1} H( \ol qo) - 2  \,\ol qo\, h( \ol yo)
	\end{align*}
	due to \textsc{Additivity}. Furthermore, $H(\delta)  = \int_0^\delta h(x) \dl x \geq \frac12 \delta h(\frac12 \delta)$.
	With that, \textsc{$h$-Moment} implies \textsc{IntegrableInf}. 
	Thus, \autoref{thm:epi} can be applied.
\end{proof}
\begin{corollary}\label{cor:nondec:onehaus}
	Assume \textsc{SampleHeineBorel}, \textsc{Polish}, \textsc{Additivity}, \textsc{InfiniteIncrease}, and \textsc{$h$-Moment}.
	Then
	\begin{equation*}
		d_\subset(M_n, M) \xrightarrow{n\to\infty}_{\ms{a.s.}} 0
		\eqfs
	\end{equation*}
\end{corollary}
\begin{proof}
	We check the conditions of \autoref{thm:consistency}. \textsc{SampleHeineBorel} and \textsc{Polish} are assumptions of the corollary. \textsc{LowerSemiContinuity} and \textsc{IntegrableInf} are shown in the proof of \autoref{cor:nondec:epi}.
	Following that proof, we find, due to \textsc{Additivity},
	\begin{align*}
		\abs{H(\ol yq) - H(\ol yo)} &\leq b\, \ol qo \br{h(\ol qo) + h(\ol yo)}\eqcm
		\\
		H(\ol yq) - H(\ol yo) &\geq 	b^{-1} H( \ol qo) - 2 \,\ol qo\, h( \ol yo)\eqcm\\
		H(\delta)  &\geq \frac12 \delta h\brOf{\frac12 \delta}
		\eqfs
	\end{align*}
	The first inequality together with \textsc{$h$-Moment} implies \textsc{UpperBound}.
	For \textsc{LowerBound}, we use the second inequality. We set $\psi^+(\delta) = 	b^{-1} H(\delta)$, $\psi^-(\delta) = 2 \delta$, $\mf a^+=\mf a_n^+ = 1$, $\mf a^{-} = \Ex{h(\ol Yo)}$, and $\mf a^-_n = \frac1n \sum_{i=1}^n h(\ol {Y_i}o)$ with $\mf a^-_n\xrightarrow{n\to\infty}_{\ms{a.s.}}\mf a^-$ due to \textsc{$h$-Moment}. Because of the third inequality, $\psi^+(\delta)/\psi^-(\delta) \geq \frac14 b^{-1} h(\frac12 \delta) \xrightarrow{\delta\to\infty} \infty$ by \textsc{InfiniteIncrease}.
\end{proof}
\section{Strong Laws for $\alpha$-Fréchet Mean Sets}\label{sec:power}
Let $(\mc Q, d)$ be a metric space. Let $(\Omega, \Sigma, \Pr)$ be a probability space that is silently underlying all random variables in this section. Let $Y$ be a random variable with values in $\mc Y$.
Let $\alpha > 0$. Fix an arbitrary element $o\in\mc Q$. Denote the $\alpha$-Fréchet mean set of $Y$ as $M = \argmin_{q\in\mc Q} \Ex{\ol Yq^\alpha - \ol Yo^\alpha}$.
Let $Y_1, \dots, Y_n$ be independent random variables with the same distribution as $Y$.
Choose $(\epsilon_n)_{n\in\N}\subset [0,\infty)$ with $\epsilon_n \xrightarrow{n\to\infty}0$. Set $M_n =\epsilon_n\text{-}\argmin_{q\in\mc Q} \frac1n \sum_{i=1}^n (\ol {Y_i}q^\alpha - \ol{Y_i}o^\alpha)$.
\begin{corollary}\label{cor:cons_da}
	Let $\alpha > 1$.
	Assume $\Ex{\ol Yo^{\alpha-1}} < \infty$ and \textsc{Polish}.	
	\begin{enumerate}[label=(\roman*)]
	\item Then $\outerlim_{n\to\infty}\, M_n \subset M$ almost surely.
	\item Additionally, assume \textsc{SampleHeineBorel}.
		Then $d_\subset(M_n, M) \xrightarrow{n\to\infty}_{\ms{a.s.}} 0$.
	\end{enumerate}	
\end{corollary}
\begin{proof}
	Set $h(x) = \alpha x^{\alpha-1}$. This function is non-decreasing, fulfills \textsc{Additivity} with $b = 2^{\alpha-1}$ and \textsc{InfiniteIncrease}.
	Due to $\Ex{\ol Yo^{\alpha-1}} < \infty$, \textsc{$h$-Moment} is fulfilled. Furthermore, $H(x) = x^\alpha$. Thus, \autoref{cor:nondec:onehaus} and \autoref{cor:nondec:epi} imply the claims.
\end{proof}
\begin{corollary}\label{cor:median}
	Let $\alpha \in (0,1]$. Assume \textsc{Polish}.
	\begin{enumerate}[label=(\roman*)]
	\item Then $\outerlim_{n\to\infty}\, M_n \subset M$ almost surely.
	\item Additionally, assume \textsc{SampleHeineBorel}.
		Then $d_\subset(M_n, M) \xrightarrow{n\to\infty}_{\ms{a.s.}} 0$.
	\end{enumerate}	
\end{corollary}
\begin{proof}
First, consider the case $\alpha=1$.
Apply \autoref{lmm:nondec:existence} (appendix) on $\ol Yo$ to obtain a function $h \colon [0,\infty) \to [0,\infty)$ which is strictly increasing, continuous, concave, fulfills \textsc{InfiniteIncrease}, and $\Ex{h(\ol Yo)} < \infty$. Concavity implies \textsc{Additivity} with $b=2$.
As its derivative is strictly increasing, $H(x) = \int_0^x h(t) \dl t$ is convex and strictly increasing. Thus, $H$ has an inverse $H^{-1}$ and $H^{-1}$ is concave. 
This implies that $d_H(q,p) = H^{-1}(\ol qp)$ is a metric. 

As $H^{-1}$ is concave, there are $u_0, u_1 \in[0,\infty)$ such that $H^{-1}(x) \leq u_0 + u_1 x$ for all $x\geq0$.
As $h$ is concave, there are $v_0, v_1\in[0,\infty)$ such that $h(u_0 + u_1 x) \leq v_0 + v_1 h(x)$ for all $x\geq 0$.
Thus, $\Ex{h(d_H(Y, o))} = \Ex{h(H^{-1} (\ol Yo))} \leq v_0 + v_1 \Ex{h(\ol Yo)} < \infty$.
Hence, \textsc{$h$-Moment} is true for the metric $d_H$.

Moreover, \textsc{Polish} and \textsc{HeineBorel}-type properties of $(\mc Q, d)$ are preserved in $(\mc Q, d_H)$, as $H^{-1}$ is strictly increasing, concave, and continuous, with $H^{-1}(0) = 0$ and $H^{-1}(\delta)\xrightarrow{\delta\to\infty}\infty$, and thus, the properties boundedness, compactness, separability, and completeness coincide for $d$ and $d_H$.
Applying \autoref{cor:nondec:epi} and \autoref{cor:nondec:onehaus} on the minimizers of $\Ex{H(d_H(Y, q)) - H(d_H(Y, o))}  = \Ex{\ol  Yq - \ol Yo}$ now yields the claims for $\alpha=1$.

For $\alpha \in (0,1)$ just note, that $\tilde d(q,p) = d(q,p)^\alpha$ is a metric, which preserves \textsc{Polish} and \textsc{HeineBorel}-type properties, and apply the result for $\alpha = 1$ on $\tilde d$.
\end{proof}
\begin{remark}
	 For convergence, we need the $(\alpha-1)$-moment to be finite in the case of $\alpha \geq 1$. \cite[Corollary 5]{schoetz19} shows that in metric spaces with nonnegative curvature the typical parametric rate of convergence $n^{-\frac12}$ is obtained for $\alpha$-Fréchet means assuming the $2(\alpha-1)$-moment to be finite in the case of $\alpha\in[1,2]$ under some further conditions.
\end{remark}
\begin{appendix}
\section{Example: The Set of Medians}\label{sec:median}
Let $s\in\N$. Consider the metric space $(\R^s, d_1)$, where $d_1(q,p) = \normof{q-p}_1 = \sum_{j=1}^s \abs{q_j-p_j}$. The power Fréchet mean with $\alpha=1$ in this space is equivalent to the standard ($\alpha=2$) Fréchet mean in $(\R^s, d_1^{\frac12})$. For $s=1$ it is equal to the median.
Let $Y = (Y^1,\dots, Y^s)$ be a random vector in $\R^s$ such that $\Pr(Y^k=0)=\Pr(Y^k=1)=\frac12$ for $k=1,\dots, s$ and $Y^1,\dots, Y^s$ are independent. Let $Y_1, Y_2,\dots$ be independent and identically distributed copies of $Y$. Let $M = \argmin_{q\in\R^s}\Ex{d_1(Y, q)}$ be the Fréchet mean set of $Y$ and $M_n = \epsilon_n\text{-}\argmin_{q\in\R^s}\frac1n\sum_{i=1}^n d_1(Y_i, q)$ its sample version.

\subsection{No Convergence in Hausdorff Distance}
First consider the case $s=1$ and $\epsilon_n=0$.
As $s=1$,  $M = \argmin_{q\in\R} \Ex{\abs{Y-q}}$ is the median of $Y$, which is $M = [0,1]$ as $2\Ex{\abs{Y-q}} = \abs{1-q}+\abs{q}$ achieves its minimal value $1$ precisely for all $q\in[0,1]$.
Define $p_n := \frac1n \sum_{i=1}^n Y_i$. Then the empirical objective function is $F_n(q) := p_n\abs{1-q}+(1-p_n)\abs{q}$, i.e., the sample Fréchet mean set is $M_n = \argmin_{q\in\R} F_n(q)$.
If $n$ is odd, then either $M_n = \{0\}$ or $M_n = \{1\}$ holds. The same is true for an even value of $n$ except when $p_n=\frac12$, in which case $M_n=[0,1]$.
Thus, $d_{\subset}(M_n, M) = 0$, but $d_{\Hausdorff}(M_n, M)$ does not converge almost surely.

\subsection{The Outer Limit as a Strict Subset}
Next, we keep $\epsilon_n=0$, but consider the value of $\outerlim_{n\to\infty} M_n$ in a multi-dimensional setting, i.e., $s\in\N$, as this yields a potentially surprising result:
By \autoref{lmm:products}, $M$ is just the Cartesian product of the median sets in each dimension, i.e., $M = [0,1]^s$ (this is not to be confused with the geometric median, which is the Fréchet mean with respect to the square root of the Euclidean norm). Similarly, $M_n = \bigtimes_{k=1}^s M_n^k$ decomposes into the sample Fréchet mean sets $M_n^k$ of each dimension $k=1,\dots,s$. It holds $M_n^k = [0,1]$ if and only if the respective value of $p_n^k := \frac1n \sum_{i=1}^n Y_i^k$ is equal to $\frac12$, i.e., if and only if the symmetric simple random walk $S_n^k := \sum_{i=1}^n (2Y_i^k-1)$ hits $0$. Let $N = \#\Set{n\in\N\given S_n^1=\dots=S_n^s=0}$. Let $A\subset \R$ and $B_n\subset \R$ for all $n\in\N$. If $A\subset B_n$ for infinitely many $n$, then $A \subset \outerlim_{n\to\infty} B_n$. Thus, we want to know whether $N$ is finite or infinite. This is answered by \textit{Pólya's Recurrence Theorem} \cite{polya21}. It implies that for $s \in \{1, 2\}$, $N = \infty$ almost surely. Furthermore, if $s\geq 3$, then $N < \infty$ almost surely.
To find which points are not element of the outer limit of sample Fréchet mean sets, note following fact: For an open subset $A\subset \R$, if $A \subset \R \setminus B_n$ for all but finitely many $n$, then $A \subset \R \setminus \outerlim_{n\to\infty} B_n$.
We conclude, that in general a vector $x \in [0,1]^s$ is an element of $\outerlim_{n\to\infty} M_n$ if and only if at most two entries are not in $\{0,1\}$, i.e., almost surely
\begin{equation*}
	\outerlim_{n\to\infty} M_n = \Set{(x_1,\dots,x_s) \in [0,1]^s \given \#\{k\in\{1,\dots,s\} | x_k \in (0,1)\} \leq 2}
	\eqfs
\end{equation*}
Thus, $\outerlim_{n\to\infty} M_n = M$ for $s\in\{1,2\}$ and $\outerlim_{n\to\infty} M_n \subsetneq M$ for $s \geq 3$.

\subsection{Convergence in Hausdorff Distance}
Lastly, we use the setting $s=1$ and $\epsilon_n\in[0,\infty)$, where we want to find $\epsilon_n$ such that $[0,1] \subset M_n$.
At least one of $0$ and $1$ is a minimizer of $F_n(q)=p_n\abs{1-q}+(1-p_n)\abs{q}$ and the $F_n(q)$ is linear on $[0,1]$. Thus, $[0,1] \subset M_n$ if and only if  $\epsilon_n \geq |F_n(0)-F_n(1)|$. This is equivalent to 
 $|p_n-\frac12| < \frac12\epsilon_n$.
By Markov's inequality
\begin{equation*}
	\PrOf{\abs{p_n-\frac12} \geq \frac12\epsilon_n} \leq \frac{n^{-3} \Ex*{(Y-\frac12)^4}}{2^{-4}\epsilon_n^4}
	\eqfs
\end{equation*}
For $\epsilon_n = n^{-\frac14}$, we obtain
\begin{equation*}
	\sum_{n=1}^\infty \PrOf{\abs{p_n-\frac12} \geq \frac12\epsilon_n}  \leq \sum_{n=1}^\infty n^{-2} < \infty
	\eqfs
\end{equation*}
The Borel--Cantelli lemma implies that almost surely and for all $n$ large enough,  $|p_n-\frac12| < \frac12\epsilon_n$ and thus, $[0,1] \subset M_n$. 
Together with \autoref{cor:median}, we obtain $d_{\Hausdorff}(M_n, M) \xrightarrow{n\to\infty}_{\ms{a.s.}} 0$.
\section{Alternative Route to One-Sided Hausdorff Convergence}\label{sec:alt}
In this section, we show an alternative proof of a strong law of large numbers for generalized Fréchet mean sets in one-sided Hausdorff distance. Although the final result, \autoref{thm:alt:consistency}, is weaker than \autoref{thm:consistency}, it is very illustrative to follow this line of proof:
In contrast to the arguments in the main part of the article, it does not rely on the powerful result \cite[Theorem 1.1]{korf01}, which seems to be rather complex to prove. Instead our reasoning here is simpler and more self-contained.
Furthermore, a comparison between convergence in outer limit and in one-sided Hausdorff distance seems more natural in view of the deterministic results \autoref{thm:convOfMini} and \autoref{thm:alt:convOfMini}, and the stochastic results \autoref{thm:epi} and \autoref{thm:alt:consistency}.
\subsection{Convergence of Minimizer Sets of Deterministic Functions}\label{ssec:alt:det}
Let $(\mc Q, d)$ be a metric space. For $A\subset \mc Q$ and $\delta>0$, denote $\ball_\delta(A) = \bigcup_{x\in A} \ball_\delta(x)$. 
\begin{definition} \mbox{ }
	\begin{enumerate}[label=(\roman*)]
	\item 
		Let $f,f_n \colon \MS\to\R$, $n\in\N$.
		The sequence $(f_n)_{n\in\N}$  \emph{converges} to $f$ \emph{uniformly on bounded sets} if and only if for every $B\subset \MS$ with $\diam(B)<\infty$,
		\begin{equation*}
		\lim_{n\to\infty} \sup_{x\in B}\abs{f_n(x) - f(x)} =0
		\eqfs
		\end{equation*}
		We then write $f_n\ubsconv{n}f$.
	\item 
		A sequence $(B_n)_{n\in\N}$ of sets $B_n \subset \MS$ is called \emph{eventually bounded} if and only if
		\begin{equation*}
		\limsup_{n\to\infty} \diam\brOf{\bigcup_{k=n}^\infty B_k} < \infty
		\eqfs
		\end{equation*}
	\item 
		A function $f$ has \emph{approachable minimizers} if and only if for all $\epsilon >0$ there is a $\delta>0$ such that
		$\delta\text{-}\argmin f \subset B_\epsilon(\argmin f)$.
	\end{enumerate}	
\end{definition}
The last definition directly implies that $d_\subset(\delta\text{-}\argmin f, \argmin f) \xrightarrow{\delta\to0} 0$ is equivalent to $f$ having approachable minimizers. Furthermore, if $f$ has approachable minimizers, then $\argmin f \neq \emptyset$, as for every $\delta>0$ the set $\delta\text{-}\argmin f$ is non-empty, but $\ball_\epsilon(\emptyset)=\emptyset$.
\begin{theorem}\label{thm:alt:convOfMini}
	Let $f,f_n \colon \MS\to\R$. Let $(\epsilon_n)_{n\in\N}\subset [0,\infty)$ with $\epsilon_n \xrightarrow{n\to\infty}0$. 
	Assume $f$ has approachable minimizers, $f_n \xrightarrow{n\to\infty}_{\ubs} f$, and $(\epsilon_n\text{-}\argmin f_n)_{n\in\N}$ is eventually bounded.
	Then
	\begin{equation*}
		d_\subset(\epsilon_n\text{-}\argmin f_n, \argmin f)\xrightarrow{n\to\infty}
		0
	\end{equation*}
	and
	\begin{equation*}
		\inf f_n \xrightarrow{n\to\infty} \inf f
		\eqfs
	\end{equation*}
\end{theorem}
\begin{proof}
		Let $\epsilon >0$. As $f$ has approachable minimizers, there is $\delta > 0$ such that $(3\delta)\text{-}\argmin f \subset \ball_\epsilon(\argmin f)$; also $\argmin f \neq \emptyset$. Let $y\in\argmin f$. As $f_n(y) \xrightarrow{n\to\infty} f(y)$, there is $n_1\in\N$ such that $\inf f_n \leq \inf f + \delta$ for all $n \geq n_1$.		
		As $\epsilon_n \xrightarrow{n\to\infty} 0$, there is $n_2\in\N$ such that $\epsilon_n \leq \delta$ for all $n \geq n_2$.
		As $(\epsilon_n\text{-}\argmin f_n)_{n\in\N}$ is eventually bounded, there is $n_3\in\N$ such that $\diam(B) < \infty$ for $B = \bigcup_{n\geq n_3} \epsilon_n\text{-}\argmin f_n$.
		As $f_n \xrightarrow{n\to\infty}_{\ubs} f$ there is $n_4$ such that $\sup_{x\in B}\abs{f_n(x) - f(x)} \leq \delta$.
		Let $n \geq \max(n_1, n_2, n_3, n_4)$ and $x \in \epsilon_n\text{-}\argmin f_n$. Then
		\begin{equation*}
			f(x) \leq f_n(x) + \delta \leq \inf f_n + 2\delta \leq \inf f + 3\delta\eqfs
		\end{equation*}
		Thus, $x \in (3\delta)\text{-}\argmin f$. By the choice of $\epsilon$ and $\delta$, we obtain $\epsilon_n\text{-}\argmin f_n \subset \ball_\epsilon(\argmin f)$ or equivalently $d_{\subset}(\epsilon_n\text{-}\argmin f_n, \argmin f) \leq \epsilon$.
		
		Finally, we show the convergence of the infima. We already know $\inf f_n \leq \inf f + \epsilon$ for all $\epsilon > 0$ and $n$ large enough.
		If $\inf f_n \xrightarrow{n\to\infty} \inf f$ does not hold, there is a sequence $x_n \in \epsilon_n\text{-}\argmin f_n$ and $\epsilon > 0$ such that $f_n(x_n) < \inf f - \epsilon$ for all $n$ large enough. As before, because of eventual boundedness and uniform convergence on bounded sets, we have $\sup_{k\in\N}\abs{f_n(x_k) - f(x_k)} \xrightarrow{n\to\infty} 0$. Therefore, for all $\epsilon >0$ we have $f(x_n) \leq f_n(x_n) + \epsilon$ for $n$ large enough, which contradicts $f_n(x_n) < \inf f - \epsilon$.
\end{proof}
In the following, we construct examples to show that none of the conditions for one-sided Hausdorff convergence can be dropped.
\begin{example}\label{exa:alt:necessary}
\mbox{ }
\begin{enumerate}[label=(\roman*)]
	\item 
	Let $f,f_n\colon \N_0 \to \R$, $f_n = 1-\indOf{\cb{0,n}}$, $f = 1-\indOf{\cb{0}}$, $d(i,j) = 1$ for $i\neq j$.
	It holds that $f$ is continuous and has approachable minimizers, and the sequence of nonempty sets $\argmin f_n = \cb{0,n}$ is eventually bounded, as $\diam(A)\leq1$ for every $A\subset\N_0$. Furthermore, $f_n$ converges to $f$ uniformly on compact sets, which are exactly the finite subsets of $\N_0$, but not uniformly on bounded sets like $\N_0$ itself.
	There is a subsequence of minimizers $x_n=n\in \argmin f_n$ that is always bounded away from $0$, the minimizer of $f$.
	This shows that uniform convergence on compact sets (instead of bounded sets) is not enough.
	\item
	As above, let $f, f_n\colon \N_0 \to \R$, $f_n = 1-\indOf{\cb{0,n}}$, $f = 1-\indOf{\cb{0}}$, but define $d(i,j) = |i-j|$.
	It holds that $f$ is continuous and has approachable minimizers, and $f_n\ubsconv{n}f$, but the sequence of nonempty sets $\argmin f_n = \cb{0,n}$ is not eventually bounded. 
	Again, there is a subsequence of minimizers $x_n=n\in \argmin f_n$ that is always bounded away from $0$, the minimizer of $f$.
	This shows that eventual boundedness of minimizer sets cannot be dropped.
	\item
	Let $f,f_n\colon \N_0 \to \R$, $f(0) = 0$, $f(i) = \frac1i$, $f_n(i) = f(i)\indOfEvent{i < n}$, and set $d(i,j) = 1$  for $i\neq j$.
	It holds that $f$ is continuous, but $f$ does not have approachable minimizers. The sequence of nonempty sets $\argmin f_n = \cb{0,n, n+1, \dots}$ is eventually bounded and $f_n\ubsconv{n}f$.
	There is a subsequence of minimizers $x_n=n\in \argmin f_n$ that is always bounded away from $0$, the minimizer of $f$.
	This shows that approachability of minimizers of $f$ cannot be dropped.
\end{enumerate}
\end{example}
\subsection{Strong Laws for $\mf c$-Fréchet Mean Sets}\label{ssec:alt:gen}
Let $(\mc Q, d)$ be a metric space, the descriptor space. Let $\mc Y$ be a set, the data space. Let $\mf c\colon \mc Y \times\mc Q \to \R$ be a function, the cost function. Let $(\Omega, \Sigma, \Pr)$ be a probability space that is silently underlying all random variables in this section. Let $Y$ be a random variable with values in $\mc Y$. Denote the $\mf c$-Fréchet mean set of $Y$ as $M = \argmin_{q\in\mc Q} \Ex{\mf c(Y, q)}$.
Let $Y_1, \dots, Y_n$ be independent random variables with the same distribution as $Y$.
Choose $(\epsilon_n)_{n\in\N}\subset [0,\infty)$ with $\epsilon_n \xrightarrow{n\to\infty}0$. Set $M_n = \epsilon_n\text{-}\argmin_{q\in\mc Q} \frac1n \sum_{i=1}^n \mf c(Y_i, q)$.
\begin{assumptions}\mbox{ }
\begin{itemize}
	\item \textsc{Continuity}:  The function $q \mapsto \mf c(Y, q)$ is continuous almost surely.
	\end{itemize}
\end{assumptions}
\begin{theorem}\label{thm:alt:consistency}
	Assume \textsc{HeineBorel}, \textsc{Continuity}, \textsc{UpperBound}, and \textsc{LowerBound}. Then
	\begin{equation*}
		d_\subset(M_n, M) \xrightarrow{n\to\infty}_{\ms{a.s.}} 0
		\eqfs
	\end{equation*}
\end{theorem}
\begin{proof}
Define $F(q) = \Ex{\mf c(Y, q)}$, $F_n(q) = \frac1n\sum_{i=1}^n \mf c(Y_i, q)$.  The proof consists of following steps:
\begin{enumerate}
\item Show that $F_n \xrightarrow{n\to\infty}_{\ubs} F$ almost surely.
\item Reduction to a bounded set.
\item Show that $F$ has approachable minimizers.
\item Show that $M_n$ is eventually bounded.
\item Apply \autoref{thm:alt:convOfMini}.
\end{enumerate}
\underline{\smash{Step 1.}}
To show uniform convergence on bounded sets, we will use the uniform law of large numbers, \autoref{thm:ULLN} below.
Let $B\subset\mc Q$ be a bounded set.
By \textsc{HeineBorel}, $\overline{B}$ is compact.
By \textsc{Continuity}, $q \mapsto \mf c(Y, q)$ is almost surely continuous.
By \textsc{UpperBound}, $\Ex{ \sup_{q\in B}\abs{\mf c(Y, q)}} < \infty$. Thus, \autoref{thm:ULLN} implies that $q \mapsto F(q)$ is continuous and
\begin{equation*}
	\sup_{q\in B}  \abs{F_n(q)-F(q)} \xrightarrow{n\to\infty}_{\ms{a.s.}}0
	\eqfs
\end{equation*}
Fix an arbitrary element $o\in\mc Q$. For all bounded sets $B$, there is $\delta \in\N$ such that $B \subset \ball_\delta(o)$. By the previous considerations, uniform convergence holds almost surely for all $(\ball_{\delta}(o))_{\delta\in\N}$.
Thus, $F_n \xrightarrow{n\to\infty}_{\ubs} F$ almost surely.

\noindent
\underline{\smash{Step 2.}}
Find $B_1\subset\mc Q$ and a random variable $N_1\in\N$ as in step 2 in the proof of \autoref{thm:consistency}.

\noindent
\underline{\smash{Step 3.}}
Clearly, $M \subset B_1$ is bounded. Furthermore, for all $\epsilon > 0$ small enough the set $D_\epsilon = \overline{B_1 \setminus \ball_\epsilon(M)}$ is not empty (if it is, increase $\delta$), does not contain any element of $M$ and, by \textsc{HeineBorel}, is compact. Thus, the continuous function $q\mapsto F(q)$ attains its infimum on $D_\epsilon$ where $\inf_{q\in D_\epsilon} F(q) > \inf_{q\in \mc Q} F(q)$. Take $\zeta = \min(1, \frac12(\inf_{q\in D_\epsilon} F(q) - \inf_{q\in \mc Q} F(q)))$. Then $\zeta\text{-}\argmin_{q\in\mc Q} F(q) \subset \ball_\epsilon(M)$, i.e., $F$ has approachable minimizers. 

\noindent
\underline{\smash{Step 4.}}
For $\epsilon_n < 1$ and $n \geq N_1$, it holds $M_n \subset B_1$. Thus, $(M_n)_{n\in\N}$ is eventually bounded almost surely.

\noindent
\underline{\smash{Step 5.}}
Finally, \autoref{thm:alt:convOfMini} implies $d_\subset(M_n, M) \xrightarrow{n\to\infty}_{\ms{a.s.}} 0$.
\end{proof}
\section{Auxiliary Results}\label{sec:aux}
There are many versions of uniform laws of large numbers in the literature. We state and prove one version that is tailored to our needs.
\begin{theorem}\label{thm:ULLN}
Let $(\mc Y, \Sigma_{\mc Y})$ be a measurable space and $Y$ be a random variable with values in $\mc Y$. Let $Y_1, \dots, Y_n$ be independent and have the same distribution as $Y$.
Let $(\mc Q, d)$ be a metric space and $B\subset \mc Q$ compact. 
Let $f \colon \mc Y \times B \to \R$ be such that $q \mapsto f(Y, q)$  is almost surely continuous. Assume there is a random variable $Z$ such that $\abs{f(Y, q)} \leq Z$ for all $q \in B$ with $\Ex{Z} < \infty$. Then $q \mapsto \Ex{f(Y, q)}$ is continuous and 
\begin{equation*}
	\sup_{q\in B} \abs{\frac 1n \sum_{i=1}^n f(Y_i, q) - \Ex{f(Y, q)}} \xrightarrow{n\to\infty}_{\ms{a.s.}} 0 
	\eqfs
\end{equation*}
\end{theorem}
\begin{proof}
Let $\epsilon > 0$. As $B$ is compact, there is a finite set $\cb{q_1, \dots, q_k} \subset \mc Q$ such that $B \subset \bigcup_{\ell=1}^k\ball_\epsilon(q_\ell)$. We split the supremum,
\begin{align*}
	&\sup_{q\in B} \abs{\frac 1n \sum_{i=1}^n f(Y_i, q) - \Ex{f(Y, q)}}
	\\&\leq
	\sup_{\ell \in \cb{1,\dots, k}} 
	\sup_{q\in \ball_\epsilon(q_\ell)}  
		\abs{\frac 1n \sum_{i=1}^n \br{f(Y_i, q) - f(Y_i, q_\ell)} - \Ex{f(Y, q) - f(Y, q_\ell)}} 
	\\&\quad+ 
	\sup_{\ell \in \cb{1,\dots, k}} \abs{\frac 1n \sum_{i=1}^n f(Y_i, q_\ell) - \Ex{f(Y, q_\ell)}}
	\eqfs
\end{align*}
For the second summand, by the standard strong law of large numbers applied to each $\ell\in\{1,\dots,k\}$ with $\Ex{Z} < \infty$, 
\begin{align*}
\sup_{\ell \in \cb{1,\dots, k}} \abs{\frac 1n \sum_{i=1}^n f(Y_i, q_\ell) - \Ex{f(Y, q_\ell)}} \xrightarrow{n \to \infty}_{\ms{a.s.}} 0
\eqfs
\end{align*}
For the first summand,
\begin{align*}
	&\sup_{\ell \in \cb{1,\dots, k}} 
	\sup_{q\in \ball_\epsilon(q_\ell)}  
		\abs{\frac 1n \sum_{i=1}^n \br{f(Y_i, q) - f(Y_i, q_\ell)} - \Ex{f(Y, q) - f(Y, q_\ell)}} 
	\\& \leq
	\frac 1n \sum_{i=1}^n \sup_{q,p\in B,\, \sol qp \leq \epsilon} \abs{f(Y_i, q) - f(Y_i, p)} 
	+ 
	\Ex*{\sup_{q,p\in B,\, \sol qp \leq \epsilon} \abs{f(Y, q) - f(Y, p)}} 
	\eqfs
\end{align*}
By the standard strong law of large numbers with $\Ex{Z} < \infty$, 
\begin{align*}
	\frac 1n \sum_{i=1}^n \sup_{q,p\in B,\, \sol qp \leq \epsilon} \abs{f(Y_i, q) - f(Y_i, p)} 
	&\xrightarrow{n \to \infty}_{\ms{a.s.}} 
	\Ex*{\sup_{q,p\in B,\, \sol qp \leq \epsilon} \abs{f(Y, q) - f(Y, p)} }
	\eqfs
\end{align*}
Thus,
\begin{equation}\label{eq:ulln:as_limsup}
	\PrOf{
		\limsup_{n\to\infty} \sup_{q\in B} \abs{\frac 1n \sum_{i=1}^n f(Y_i, q) - \Ex{f(Y, q)}}
		\leq
		a_\epsilon
	} = 1
	\eqcm
\end{equation}
where $a_\epsilon = 2 \Ex*{\sup_{q,p\in B,\, \sol qp \leq \epsilon} \abs{f(Y, q) - f(Y, p)} }$.
As $q \mapsto f(Y, q)$ is almost surely continuous and $B$ is compact, $q \mapsto f(Y, q)$ is almost surely uniformly continuous, i.e., for all $\delta > 0$ there is $\varepsilon(\delta, Y) > 0$ such that $\abs{f(Y,q)-f(Y,p)} \leq \delta$ for all $\ol qp \leq \varepsilon(\delta, Y)$.
As $\Ex{Z} < \infty$, we can use dominated convergence to obtain 
\begin{align*}
	\lim_{\epsilon\searrow0} 	\Ex*{\sup_{q,p\in B,\, \sol qp \leq \epsilon} \abs{f(Y, p) - f(Y, p)} }
	&=
	\Ex*{\lim_{\epsilon\searrow0}\sup_{q,p\in B,\, \sol qp \leq \epsilon} \abs{f(Y, p) - f(Y, p)} }
	=
	0
	\eqfs
\end{align*}
Thus, $a_\epsilon \xrightarrow{\epsilon\searrow0} 0$. Together with \eqref{eq:ulln:as_limsup}, this implies
\begin{equation*}
 \sup_{q\in B} \abs{\frac 1n \sum_{i=1}^n f(Y_i, q) - \Ex{f(Y, q)}} \xrightarrow{n \to \infty}_{\ms{a.s.}} 0
 \eqfs
\end{equation*}
We have also shown that $q\mapsto \Ex{f(Y, q)}$ is continuous, as $\abs{\Ex{f(Y, q)}-\Ex{f(Y, p)}} \leq a_{\sol qp}$.
\end{proof}
\begin{lemma}\label{lmm:nondec}
Let $h \colon [0,\infty) \to [0,\infty)$ be a non-decreasing function.
Define $H \colon [0,\infty) \to [0,\infty), x\mapsto \int_0^x h(t) \dl t$.
Let $x, y \geq 0$. Then
\begin{enumerate}[label=(\roman*)]
\item $\abs{H(x) - H(y)} \leq \abs{x-y}h(\max(x, y))$.
\end{enumerate}
Assume, there is $b \in [1,\infty)$ such that $h(2u) \leq b h(u)$ for all $u \geq 0$. Then
\begin{enumerate}[label=(\roman*),start=2]
\item $\frac12 h(x) + \frac12 h(y) \leq h (x+y) \leq b \br{h(x) + h(y)}$,
\item $H(\abs{x-y}) - H(x) \geq b^{-1} H(y) - 2 y h(x)$.
\end{enumerate}
\end{lemma}
\begin{proof}
\begin{enumerate}[label=(\roman*)]
\item This is a direct consequence of the mean value theorem.
\item As $h$ is non-decreasing, $\max(h(x), h(y)) \leq h(x+y) \leq \max(h(2x), h(2y))$.
	By the definition of $b$ and with $\frac12(u+v) \leq \max(u,v) \leq u+v$ for $u,v \geq 0$ the claim follows.
\item 
	First, consider the case $x\geq y$.
 	Define $f(x,y) = H(x-y)-H(x)- b^{-1} H(y) + 2 y h(x)$.
 	We want to show $f(x, y)\geq0$.
	The derivative of $f$ with respect to $y$ is 
	\begin{equation*}
		\partial_y f(x, y) = -h(x-y) - b^{-1}h(y) + 2 h(x)\eqfs
	\end{equation*} 
	By applying the first inequality of (ii) to $h(x) = h((x-y) + y)$, we obtain $\partial_y f(x, y)  \geq 0$ as $b^{-1} \leq 1$.
	Hence, $f(x, y) \geq f(x, 0) = 0$, as $H(y) = 0$.
	
	Now, consider the case $x \leq y$. Set $g(x,y) = H(y-x) - H(x) - b^{-1} H(y) + 2 y h(x)$, which yields
	\begin{equation*}
		\partial_y g(x, y) = h(y-x) - b^{-1} h(y) + 2 h(x)\eqfs
	\end{equation*} 
	By applying the second inequality of (ii) to $h(y) = h((y-x) + x)$, we obtain $\partial_y g(x, y)  \geq 0$ as $b^{-1} \leq 1$.
	Thus, $g(x, y) \geq g(x, x) = -(1 + b^{-1}) H(x) + 2 x h(x)$ as $H(0) = 0$.
	By the definition of $H$, as $h$ is non-decreasing, $H(x) \leq x h(x)$. Hence, $g(x, y) \geq 0$ as $1 + b^{-1} \leq 2$.

	Together, we have shown $H(\abs{x-y}) - H(x) - b^{-1} H(y) + 2 y h(x) \geq 0$ for all $x,y\geq0$.
\end{enumerate}
\end{proof}
\begin{lemma}\label{lmm:nondec:existence}
	Let $X$ be a random variable with values in $[0,\infty)$. Then there is a strictly increasing, continuous, and concave function $h \colon [0,\infty) \to [0,\infty)$ with $h(\delta)\xrightarrow{\delta\to\infty}\infty$ such that $\Ex{h(X)} < \infty$.
\end{lemma}
\begin{proof}
	If there is $K > 0$ such that $\Prof{X < K} = 1$ take $h(x) = x$. Now, assume that $X$ is not almost surely bounded.
	We first construct a non-decreasing function $\tilde h \colon [0,\infty) \to [0,\infty)$ such that $\tilde h (x) \xrightarrow{x\to\infty}\infty$ with $\Ex{\tilde h(X)} < \infty$.
	Then we construct a function $h$ from $\tilde h$ with all desired properties.
	
	Let $F$ be the distribution function of $X$, $F(x) = \Prof{X \leq x}$.
	Let $z_1 = 0$ and $z_{n+1} = \inf \cbOf{x\geq z_n +1 \,\big\vert\, 1-F(x) \leq \frac1n}$. As $F(x)\xrightarrow{x\to\infty} 1$, $z_n < \infty$. Furthermore, $z_{n+1}-z_n \geq 1$. Moreover, as $X$ is not almost surely bounded, $1-F(x) > 0$ for all $x\geq 0$. Set 
	\begin{align*}
		g(x) &= \sum_{n=1}^\infty (z_{n+1}-z_n)^{-1}n^{-2}\ind_{[z_n, z_{n+1})} (x)
		\eqcm
		\\
		\tilde h(x) &= \int_0^x \frac{g(t)}{1-F(t)} \dl t
		\eqfs
	\end{align*}
	Then
	\begin{align*}
		\lim_{x\to\infty} \tilde h(x) 
		= 
		\int_0^\infty \frac{g(t)}{1-F(t)} \dl t
		\geq
		\sum_{n=1}^\infty n^{-1} 
		= 
		\infty
		\eqfs
	\end{align*}
	Moreover, $\tilde h(x)$ is strictly increasing, as $g(t)\geq 0$ and $1-F(t) \geq 0$.
	The function $\tilde h$ is continuously differentiable everywhere except at point $z_n$, $n\in\N$. Thus,
	\begin{align*}
		\Ex{\tilde h(X)}
		&=
		\int_0^\infty \PrOf{\tilde h(X) > t} \dl t
		\\&=
		\int_0^\infty \PrOf{X > \tilde h^{-1}(t)} \dl t
		\\&=
		\int_0^\infty \tilde h\pr(t) \PrOf{X > t} \dl t
		\\&=
		\int_0^\infty g(t) \dl t
		\\&=
		\sum_{n=1}^\infty n^{-2}
		<
		\infty
		\eqfs
	\end{align*}
	Let $a_0 = 1$, $x_0=0$, $x_{n+1} = \inf\cbOf{x \geq x_n + a_n^{-1} \,\big\vert\, \tilde h(x) \geq n+1}$ and $a_{n+1} = (x_{n+1} - x_n)^{-1}$. Let $h \colon [0,\infty) \to [0,\infty)$ be the linear interpolation of $(x_n, n)_{n\in\N_0}$. 
	As $\tilde h(x) \xrightarrow{x\to\infty}\infty$, all $x_n$ are finite. Hence, $h(x) \xrightarrow{x\to\infty}\infty$. 
	Because of $a_n > 0$, $h$ is strictly increasing.
	Furthermore, $a_{n+1} \leq a_n$ as $x_{n+1} \geq x_n + a_n^{-1}$. As $h$ is continuous and $a_n$ is the derivative of $h$ in the interval $(x_n, x_{n+1})$, $h$ is concave. 
	Lastly, $h(x) \leq \tilde h(x) + 1$. Thus, $\Ex{h(X)} < \infty$.
\end{proof}
\begin{lemma}[Fréchet means in product spaces]\label{lmm:products}
Let $K\in\N$. Let $(\mc Q_1, d_1), \dots, (\mc Q_K, d_K)$ be metric spaces. 
Let $\alpha\geq1$.
Set $\mc Q := \bigtimes_{k=1}^K \mc  Q_k$ and $d \colon \mc Q \times \mc Q \to [0,\infty)$, $d(q,p) := (\sum_{k=1}^K d_k(q_k, p_k)^\alpha)^\frac1\alpha$. Then $(\mc Q, d)$ is a metric space. Let $Y = (Y^1, \dots, Y^K)$ be a tuple of random variables such that $Y^k$ has values in $\mc Q_k$ and $\Ex{d(Y,o)^{\alpha-1}} < \infty$ for an element $o\in\mc Q$. Let $M$ be the $\alpha$-Fréchet mean set of $Y$ in $(\mc Q, d)$, and $M^k$ be the $\alpha$-Fréchet mean set of $Y^k$ in $(\mc Q_k, d_k)$. Then $M$ is the Cartesian product of the sets $M^k$, i.e.,
\begin{equation*}
	M = \bigtimes_{k=1}^K M^k
	\eqfs
\end{equation*}
\end{lemma}
\autoref{lmm:products} can be proven by straight forward calculations.
\begin{lemma}\label{lmm:precompact}
Let $(\mc Q, d)$ be a complete metric space and $A\subset \mc Q$. Assume that all closed subsets $B\subset A$ are compact. Then $\overline A$ is compact.
\end{lemma}
\begin{proof}
	Let $(a_k)_{k\in\N}\subset \overline{A}$. We show that $(a_k)_{k\in\N}$ has a converging subsequence with limit $a \in \overline{A}$, which implies compactness of $\overline{A}$.
	
	Let $\delta_n = 2^{-n}$. Define $B_n := \mc Q \setminus \ball_{\delta_n}(\mc Q \setminus A)$. The sets $B_n$ are closed and subsets of $A$. Thus, they are compact. Let $b_k^n\in\argmin_{b \in B_n} d(b, a_k)$. Such an element exists as $B_n$ is compact. Furthermore, $d(b_k^n, a_k) \leq \delta_n$. Define the subindex sequences $(k(n,\ell))_{\ell\in\N} \subset \N$ such that $k(0, \ell) = \ell$ and $(b^n_{k(n, \ell)})_{\ell\in\N}$ is a converging subsequence of $(b^n_{k(n-1, \ell)})_{\ell\in\N}$ with limit $b^n_{k(n, \ell)} \xrightarrow{\ell\to\infty}b^n_\infty$ and $d(b^n_{k(n, \ell)}, b^n_\infty) \leq \delta_\ell$. By the triangle inequality $d(b^{n_1}_k, b^{n_2}_k) \leq d(b^{n_1}_k, a_k) + d(b^{n_2}_k, a_k) \leq \delta_{n_1}+\delta_{n_2}$. Thus, $d(b^{n_1}_\infty, b^{n_2}_\infty) \leq \delta_{n_1}+\delta_{n_2}$, which makes $(b^{n}_\infty)_{n\in\N}$ a Cauchy-sequence. Define $a$ as its limit, i.e., $b^{n}_\infty \xrightarrow{n\to\infty} a$. As $b^{n}_\infty\in B_n \subset\overline{A}$, also $a\in\overline{A}$. Finally, the triangle inequality yields
	\begin{equation*}
		d(a_{k(n, n)}, a) \leq d(a_{k(n, n)}, b^n_{k(n, n)}) + d(b^n_{k(n, n)}, b^{n}_\infty) + d(b^{n}_\infty, a) \xrightarrow{n\to\infty}0\eqcm
	\end{equation*}
	i.e., $(a_{k(n, n)})_{n\in\N}$ is subsequence of $(a_k)_{k\in\N}$ which converges in $\overline{A}$.
\end{proof}
\end{appendix}
\bibliographystyle{apalike}
\bibliography{literature}
\end{document}